\documentclass[12pt, reqno]{amsart}

\usepackage{amsmath}
\usepackage{amsfonts}
\usepackage{amssymb}
\usepackage[all]{xy}
\begin{document}
\numberwithin{equation}{section}

\def\1#1{\overline{#1}}
\def\2#1{\widetilde{#1}}
\def\3#1{\widehat{#1}}
\def\4#1{\mathbb{#1}}
\def\5#1{\frak{#1}}
\def\6#1{{\mathcal{#1}}}

\newcommand{\de}{\partial}
\newcommand{\R}{\mathbb R}
\newcommand{\Ha}{\mathbb H}
\newcommand{\al}{\alpha}
\newcommand{\tr}{\widetilde{\rho}}
\newcommand{\tz}{\widetilde{\zeta}}
\newcommand{\tk}{\widetilde{C}}
\newcommand{\tv}{\widetilde{\varphi}}
\newcommand{\hv}{\hat{\varphi}}
\newcommand{\tu}{\tilde{u}}
\newcommand{\tF}{\tilde{F}}
\newcommand{\debar}{\overline{\de}}
\newcommand{\Z}{\mathbb Z}
\newcommand{\C}{\mathbb C}
\newcommand{\Po}{\mathbb P}
\newcommand{\zbar}{\overline{z}}
\newcommand{\G}{\mathcal{G}}
\newcommand{\So}{\mathcal{S}}
\newcommand{\Ko}{\mathcal{K}}
\newcommand{\U}{\mathcal{U}}
\newcommand{\B}{\mathbb B}
\newcommand{\oB}{\overline{\mathbb B}}
\newcommand{\Cur}{\mathcal D}
\newcommand{\Dis}{\mathcal Dis}
\newcommand{\Levi}{\mathcal L}
\newcommand{\SP}{\mathcal SP}
\newcommand{\Sp}{\mathcal Q}
\newcommand{\A}{\mathcal O^{k+\alpha}(\overline{\mathbb D},\C^n)}
\newcommand{\CA}{\mathcal C^{k+\alpha}(\de{\mathbb D},\C^n)}
\newcommand{\Ma}{\mathcal M}
\newcommand{\Ac}{\mathcal O^{k+\alpha}(\overline{\mathbb D},\C^{n}\times\C^{n-1})}
\newcommand{\Acc}{\mathcal O^{k-1+\alpha}(\overline{\mathbb D},\C)}
\newcommand{\Acr}{\mathcal O^{k+\alpha}(\overline{\mathbb D},\R^{n})}
\newcommand{\Co}{\mathcal C}
\newcommand{\Hol}{{\sf Hol}(\mathbb H, \mathbb C)}
\newcommand{\Aut}{{\sf Aut}(\mathbb D)}
\newcommand{\D}{\mathbb D}
\newcommand{\oD}{\overline{\mathbb D}}
\newcommand{\oX}{\overline{X}}
\newcommand{\loc}{L^1_{\rm{loc}}}
\newcommand{\la}{\langle}
\newcommand{\ra}{\rangle}
\newcommand{\thh}{\tilde{h}}
\newcommand{\N}{\mathbb N}
\newcommand{\kd}{\kappa_D}
\newcommand{\Hr}{\mathbb H}
\newcommand{\ps}{{\sf Psh}}
\newcommand{\Hess}{{\sf Hess}}
\newcommand{\subh}{{\sf subh}}
\newcommand{\harm}{{\sf harm}}
\newcommand{\ph}{{\sf Ph}}
\newcommand{\tl}{\tilde{\lambda}}
\newcommand{\gdot}{\stackrel{\cdot}{g}}
\newcommand{\gddot}{\stackrel{\cdot\cdot}{g}}
\newcommand{\fdot}{\stackrel{\cdot}{f}}
\newcommand{\fddot}{\stackrel{\cdot\cdot}{f}}
\def\v{\varphi}
\def\Re{{\sf Re}\,}
\def\Im{{\sf Im}\,}
\def\ext{{\sf ext}\,}
\def\Lr{{\sf Lr}\,}
\def\tr{{\sf tr}\,}

\def\Label#1{\label{#1}}


\def\cn{{\C^n}}
\def\cnn{{\C^{n'}}}
\def\ocn{\2{\C^n}}
\def\ocnn{\2{\C^{n'}}}
\def\je{{\6J}}
\def\jep{{\6J}_{p,p'}}
\def\th{\tilde{h}}


\def\dist{{\rm dist}}
\def\const{{\rm const}}
\def\rk{{\rm rank\,}}
\def\id{{\sf id}}
\def\aut{{\sf aut}}
\def\Aut{{\sf Aut}}
\def\CR{{\rm CR}}
\def\GL{{\sf GL}}
\def\Re{{\sf Re}\,}
\def\Im{{\sf Im}\,}
\def\U{{\sf U}}
\def\Tang{{\rm Tang}_1(\mathbb{C}^N,0)}
\def\la{\langle}
\def\ra{\rangle}

\emergencystretch15pt \frenchspacing

\newtheorem{theorem}{Theorem}[section]
\newtheorem{lemma}[theorem]{Lemma}
\newtheorem{proposition}[theorem]{Proposition}
\newtheorem{problem}[theorem]{Problem}
\newtheorem{corollary}[theorem]{Corollary}

\theoremstyle{definition}
\newtheorem{definition}[theorem]{Definition}
\newtheorem{example}[theorem]{Example}

\theoremstyle{remark}
\newtheorem{remark}[theorem]{Remark}
\numberwithin{equation}{section}

\title[Loewner on complete hyperbolic manifolds]{Infinitesimal generators and the Loewner equation on complete hyperbolic manifolds}

\author[L. Arosio]{Leandro Arosio$^{\ddagger}$}
\address{L. Arosio: Istituto Nazionale di Alta Matematica ``Francesco Severi'', Citt\`a Universitaria, Piazzale Aldo Moro 5, 00185 Rome, Italy}
\email{arosio@altamatematica.it}
\thanks{$^{\ddagger}$Titolare di una Borsa della Fondazione Roma - Terzo Settore  bandita dall'Istituto Nazionale di Alta Matematica}

\author[F. Bracci]{Filippo Bracci}
\address{F. Bracci: Dipartimento Di Matematica\\
Universit\`{a} di Roma \textquotedblleft Tor Vergata\textquotedblright\ \\
Via Della Ricerca Scientifica 1, 00133 \\
Roma, Italy} \email{fbracci@mat.uniroma2.it}

\date\today

\begin{abstract}
We characterize infinitesimal generators on complete hyperbolic complex manifolds without any regularity assumption on the Kobayashi distance. This allows to prove a  general Loewner type equation with regularity of any order $d\in [1,+\infty ]$. Finally, based on these results, we focus on some open problems naturally arising.
\end{abstract}



\maketitle

\section{Introduction}

The classical theory of Ch. Loewner has been used and generalized in many aspects. We refer the reader to the book \cite{G-K}, and the recent survey papers \cite{ABCD, Br} for an updated account.

In the papers \cite{BCD1, Bracci-Contreras-Diaz-II}, the second named author with M. D. Contreras and S. D\'iaz-Madrigal developed a general theory of Loewner type both on the unit disc and on complex (Kobayashi) complete hyperbolic manifolds,  which relates evolution families to  Herglotz vector fields, via the Loewner  ODE. Next, in \cite{CDG} Contreras, D\'iaz-Madrigal and P. Gumenyuk fitted   the Loewner PDE in the picture, and, in \cite{ABHK}, the authors with H. Hamada and G. Kohr extended (with different methods) such results to complex complete hyperbolic manifolds.

However, the Loewner ODE theory required some technical hypotheses which, although satisfied in the most interesting cases (such as the unit ball of $\C^q$), made the theory a bit artificial in its generality. The aim of the present paper is exactly to show that such hypotheses are redundant.

In order to properly state the results, we need to give some notions. First, we recall that, given a complex manifold $M$, a holomorphic vector field $H$ on $M$ is said an {\sl infinitesimal generator} provided the Cauchy problem
$$
\begin{cases}
 \overset{\bullet}{z}(t)=H(z(t)),\\
z(0)=z_0
\end{cases}
$$
has a solution $z:[0,+\infty)\to M$ for all $z_0\in M$.

Infinitesimal generators are called this way because they generate continuous semigroups of holomorphic self-maps, see, {\sl e.g.}, \cite{Abate-generators, RS} for details.
In \cite{BCM} it is shown that if $D\subset \C^q$ is a bounded strongly convex domain with smooth boundary then a holomorphic vector field $H$ is an infinitesimal generator if and only if
\[
(dk_D)_{(z,w)}(H(z),H(w))\leq 0\quad \forall z,w\in D, z\neq w,
\]
where $k_D:D\times D\to \R^+$ is the Kobayashi distance of $D$ (see \cite{Kob} or \cite{Abate} for definition and properties), which is known to be $C^\infty$ outside the diagonal by  L. Lempert \cite{Le}.

As a first result we prove that the same characterization of infinitesimal generators holds in general. After having shown that the Kobayashi (pseudo)distance $k_M$ on a complex manifold $M$ is locally Lipschitz---and denoting by $dk_M$ its Dini directional derivative, which coincides a.e. with the usual differential---we prove  the following:

\begin{theorem}\label{autonomo}
Let $M$ be a complete hyperbolic complex manifold with Kobayashi distance $k_M$, and let $H$ be an holomorphic vector field on $M$. Then the following are equivalent:
\begin{itemize}
\item[a)] there exists $\varepsilon>0$ and a family of holomorphic mappings $(f_t\colon M\to M)_{t\in [0,\varepsilon)}$ such that $f_0(z)=z$ and
$$\frac{\de}{\de t} f_t(z)|_{t=0}= \lim_{t\to 0^+} \frac{f_t(z)-z}{t}=H(z),$$

\item[b)] for every $z\neq w\in M$
one has  $$(dk_M)_{(z,w)}(H(z),H(w))\leq 0,$$
\item[c)] $H$ is an infinitesimal generator.
\end{itemize}
\end{theorem}

This theorem allows to extend the so called ``product formula'' of S. Reich and D. Shoikhet \cite[Theorem 3]{reich-shoikhet} to complete hyperbolic manifolds and to prove that the set of infinitesimal generators of a complete hyperbolic complex manifold form a (closed) real cone (see Section \ref{due}).

With Theorem \ref{autonomo} in mind, we can give the following definition (cfr. \cite[Definition 3]{Bracci-Contreras-Diaz-II}):

\begin{definition}
Let $M$ be a complex manifold endowed with an Hermitian metric $\|.\|$, and let $d\in [1,+\infty]$.
A \textit{weak holomorphic vector field} of order $d\geq 1$ on $M$ is a mapping
$$G\colon M\times[0,+\infty)\to TM$$
satisfying
\begin{itemize}
\item[WHVF1.] for all $z\in M$ the map $t\mapsto G(z,t)$ is measurable,
\item[WHVF2.] for all  $t\geq 0$ the map $z\mapsto G(z,t)$ is a  holomorphic vector field,
\item[WHVF3.] for any compact set $K\subset M$ and for any $T>0$ there exists a non-negative function $c_{K,T}\in L^d([0,T],\R)$ such that $$\|G(z,t)\|\leq c_{K,T}(t),\quad z\in K, 0\leq t\leq T.$$
\end{itemize}
A \textit{Herglotz vector field} of order $d\geq 1$ on $M$ is a weak holomorphic vector field of order $d\geq 1$ such that $M\ni z\mapsto G(z,t)$ is an infinitesimal generator for a.e. fixed $t\in [0,+\infty)$.
\end{definition}

Now we are going to define evolution families:

\begin{definition}\label{def-ev}
Let $M$ be a complex manifold endowed with an Hermitian metric and let $d_M$ denote the associated integrated distance. A family $(\varphi_{s,t})_{0\leq s\leq t<+\infty}$ of holomorphic self-maps of $M$  is an {\sl evolution family of order $d\in [1,+\infty]$} (in short, an {\sl $L^d$-evolution family}) if

\begin{enumerate}
\item[EF1.] $\varphi_{s,s}={\sf id}_{M},$

\item[EF2.] $\varphi_{s,t}=\varphi_{u,t}\circ\varphi_{s,u}$ for all $0\leq
s\leq u\leq t<+\infty,$

\item[EF3.] for any compact subset $K\subset M$ and for any $T>0$ there exists a
non-negative function $k_{K,T}\in L^{d}([0,T],\mathbb{R})$
such that for all $0\leq s\leq u\leq t\leq T$ and for all $z\in K$,
\[
d_M(\varphi_{s,u}(z),\varphi_{s,t}(z))\leq\int_{u}^{t}k_{K,T}(\xi)d\xi.
\]
\end{enumerate}
\end{definition}

In \cite{Bracci-Contreras-Diaz-II} it has been proved the following result:

\begin{theorem}\cite{Bracci-Contreras-Diaz-II}
Let $M$ be a complete hyperbolic manifold with Kobayashi distance $k_M$. Assume that $k_M\in C^1(M\times M\setminus\hbox{Diag})$. \begin{enumerate}
  \item Let  $G(z,t)$ be a Herglotz vector field of order $d\in [1,+\infty]$ on $M$. Then there exists a unique evolution family $(\v_{s,t})$ of order $d$ on $M$ such that
      \begin{equation}\label{L-ODE}
      \frac{\de \v_{s,t}(z)}{\de t}=G(\v_{s,t}(z),t) \quad\hbox{a.e.\ } t\in [s,+\infty).
      \end{equation}
  \item Let $(\v_{s,t})$ be an evolution family of order $+\infty$ on $M$. Then there exists a Herglotz vector field $G(z,t)$ of order $+\infty$ which satisfies \eqref{L-ODE}. Moreover, if $G'(z,t)$ is another weak holomorphic vector field which satisfies \eqref{L-ODE} then $G(z,t)=G'(z,t)$ for all $z\in M$ and a.e. $t\geq 0$.
\end{enumerate}
\end{theorem}

In case $M=\D$ the unit disc in $\C$, in \cite{BCD1} it is proved that the second part of the previous theorem holds even when the evolution family has order $d\in [1,+\infty]$, giving rise to a  Herglotz vector field of the same order. Such a result is based on the Berkson-Porta formula for infinitesimal generators, a tool which is not available in higher dimensions. In \cite{HKM} the same is proved for the unit ball of $\C^q$, using the Loewner PDE defined in  \cite{ABHK}.

In this  paper we prove (see Propositions \ref{A1},  \ref{A3}) the following result:

\begin{theorem}\label{main}
Let $M$ be a complete hyperbolic manifold.
\begin{enumerate}
  \item Let  $G(z,t)$ be a Herglotz vector field of order $d\in [1,+\infty]$ on $M$. Then there exists a unique evolution family $(\v_{s,t})$ of order $d$ on $M$ which satisfies \eqref{L-ODE}.
  \item Let $(\v_{s,t})$ be an evolution family of order $d\in [1,+\infty]$ on $M$. Then there exists a Herglotz vector field $G(z,t)$ of order $d$ which satisfies \eqref{L-ODE}. Moreover, if $G'(z,t)$ is another weak holomorphic vector field which satisfies \eqref{L-ODE} then $G(z,t)=G'(z,t)$ for all $z\in M$ and a.e. $t\geq 0$.
\end{enumerate}
\end{theorem}

We end up the paper with a section of natural open problems deriving from what we explained before.

\section{Infinitesimal generators on complete hyperbolic manifolds}\label{due}

\subsection{Regularity of the Kobayashi distance} We first recall some definition from analysis.
A function on a manifold is said to be {\sl locally Lipschitz} if it is locally Lipschitz on one---and hence any---chart. We show that the Kobayashi pseudodistance on a complex manifold is locally Lipschitz. First we recall the following estimate  (for a proof, see, {\sl e.g.} \cite{Abate}).

\begin{lemma}\label{abate}
Let $M$ be a complex manifold, $U\subset M$ an open domain and $\psi\colon U\to \B^q$ a biholomorphism from $U$ to the open ball $\B^q\subset \C^q$. Then for every compact subset $K\subset U$ there exists $C_K>0$ such that $$k_M(z,w)\leq C_K\|\psi(z)-\psi(w)\|,\quad z,w\in K.$$
\end{lemma}

\begin{proposition}\label{loclip}
Let $M$ be a complex manifold.  Then the Kobayashi pseudodistance $$k_M\colon M\times M\to \R^+$$ is locally Lipschitz.

\end{proposition}
\begin{proof} Let $\psi\colon U\to \B^q$ and $\varphi\colon V\to \B^q$   be two  biholomorphic coordinate charts. We show that the mapping $$k_M\circ (\psi^{-1},\varphi^{-1})\colon \B^q\times\B^q\to \R^+$$ is locally Lipschitz. Let $K\subset U$ and $H\subset V$ be two compact subsets.
Let $(z,w),(z',w')$ be in $K\times H$. Then by the triangular inequality and Lemma \ref{abate} we have
\begin{equation*}
\begin{split}
|k_M(z,w)-k_M(z',w')|&=|k_M(z,w)-k_M(z',w)+k_M(z',w)-k_M(z',w')|\\&\leq |k_M(z,w)-k_M(z',w)|+|k_M(z',w)-k_M(z',w')|\\&\leq
k_M(z,z')+k_M(w,w')\leq C_K\|\psi(z)-\psi(z')\|+C_H\|\varphi(w)-\varphi(w')\|,
\end{split}
\end{equation*}
and we are done.
\end{proof}

\begin{definition}
Let $f$ be a real-valued function defined on an interval $I=[t_0,a)$. The (lower) \textit{Dini derivative} is defined as $$\underbar{D}f(t):= \liminf_{h\to 0+}\frac{f(t+h)-f(t)}{h}.$$
Let  $\Omega$ be a domain in $\R^q$, and let $f\colon \Omega\to \R$ be a locally Lipschitz function. If $x\in \Omega$ and  $v\in \R^q$, then the \textit{directional Dini derivative} is defined as
$$\underbar{D}f(x,v):= \liminf_{h\to 0+}\frac{f(x+hv)-f(x)}{h}.$$
\end{definition}

\begin{lemma}
Let  $\Omega$ be a domain in $\R^q$, and let $f\colon \Omega\to \R$ be a locally Lipschitz function. Let $\gamma\colon [0,\varepsilon)\to \Omega$ be a mapping such that $\gamma(0)=x$ and  $\frac{d}{d t}\gamma(0)=v$. Then $$ \underbar{D} f(\gamma (0))= \underbar{D}f(x,v).$$
\end{lemma}
\begin{proof}
We claim  that $|f(\gamma(t))-f(x+tv)|=o(t)$. Indeed
  there exist $C>0$ such that $$|f(\gamma(t))-f(x+tv)|\leq C\|\gamma(t)-(x+tv)\|,$$ and the claim follows since  $\frac{d}{d t}\gamma(0)=v$.

Thus
\begin{align}
\liminf_{t\to 0+}\frac{f(\gamma(t))-f(x)}{t}
&=\liminf_{t\to 0+}\frac{ f(\gamma(t))-f(x+tv) + f(x+tv)-f(x)}{t}\nonumber\\
&=\lim_{t\to 0+}\frac{f(\gamma(t))-f(x+tv)}{t}+ \liminf_{t\to 0+} \frac{f(x+tv)-f(x)}{t}\nonumber\\
&= \liminf_{t\to 0+} \frac{f(x+tv)-f(x)}{t}.\nonumber
\end{align}

\end{proof}

If $f:M \to \R$ is  locally Lipschitz  then, outside a set of zero measure (with respect to the Lebesgue measure on $M$), the differential $df_z$ exists and coincides with the Dini partial derivative $\underbar{D}f(z,\cdot)$. Therefore, it is natural to simply denote  $df_z:=\underbar{D}f(z,\cdot)$.

In particular, for what we have seen,  the Dini partial derivative $dk_M$ of the Kobayashi distance is well defined for any complex manifold $M$.

\subsection{Characterization of infinitesimal generators} Now we are in good shape to give the proof of our first result:
\begin{proof}[Proof of Theorem \ref{autonomo}]
a) $\Longrightarrow$ b)
Let  $z\neq w\in M$ and define  $\gamma_z(t):=f_t(z)$ and $\gamma_w(t):= f_t(w)$ for all $t\in [0,\varepsilon)$. Since holomorphic self-mappings of $M$ contract the Kobayashi distance, one has $ k_M(\gamma_z(t),\gamma_w(t))\leq k_M(z,w)$ for all $t>0$, thus taking the liminf of the incremental ratio as $t\to 0^+$ we get $$dk_M((z,w),(H(z),H(w)))\leq 0.$$

b) $\Longrightarrow$ c)
Let $z\in M$ and  let $f_{t}(z)$ be the maximal solution with escaping time $I(z)>0$ which solves the Cauchy problem
\begin{equation}\label{cauchy}
\begin{cases}
\frac{\de}{\de t} f_t(z)= H(f_{t}(z)),\quad t\in [0,I(z))\\
f_0(z)=z.
\end{cases}
\end{equation}
Fix $z\neq w\in M$.
Let $J:= [0,I(z))\cap [0,I(w))$ and define the continuous real valued function $$h(t) := k_M(f_{t}(z), f_{t}(w)).$$ Since $h$ is continuous and $\underbar{D}h(t)\leq 0$ for all $t\in J$, one has that $h$ is non-increasing. If it were $I(z)<I(w)$, since $M$ is complete hyperbolic and $$\{f_t(w)\}_{t\in [0,I(z)]}\subset\subset M,$$ we would have $$+\infty=\limsup_{t\to I(z)}k_M(f_t(z),f_t(w))\leq k_M(z,w)<+\infty,$$ a contradiction. Therefore $I(z)\geq I(w)$, and similarly $I(w)\geq I(z)$.

Since the Cauchy problem is autonomous, this implies that $I(z)=+\infty$ for all $z\in X$, which proves c).

c) $\Longrightarrow$ a) follows from  the holomorphic flow-box theorem (see for example \cite{Hormander}).

\end{proof}

\subsection{The product formula}

For a bounded convex domain of a complex Banach space, the following ``product formula'' is proved in \cite[Theorem 3]{reich-shoikhet} (see also \cite[p. 254]{AMR} for the smooth case).

\begin{proposition}\label{product}
Let $M$ be a complete hyperbolic complex manifold, and let $H$ be an holomorphic vector field on $M$.
Suppose there exists $\lambda>0$ and a family of holomorphic mappings $(f_t\colon M\to M)_{t\in [0,\lambda)}$ converging uniformly on compacta to  $f_0(z)=z$ as $t\to 0+$, and  such that
for all $w\in M$ one has, in a local coordinate chart $w\in W\to \C^q$,
$$\lim_{t\to 0+}\frac{f_t(z)-z}{t}=H(z)$$ uniformly in a neighborhood of $w$.
Then $H$ is an infinitesimal generator and the semigroup $(\varphi_t)$ associated to $H$ satisfies the following ``product formula'': $$\varphi_t=\lim_{m\to\infty} (f_{t/m})^{\circ m},$$  where the limit is uniform on compacta of $M$.
\end{proposition}
\begin{proof}
Fix $w\in M$. Let $\psi\colon U\to \mathbb{B}^q$  be a biholomorphic coordinate chart relatively compact in $M$ and centered at $w$, {\sl i.e.},  $\psi(w)=0$. Let  $\frac{1}{2}U:=\psi^{-1}(\frac{1}{2}\mathbb{B}^q).$ In the following, as customary, we identify   $U$ and $\mathbb{B}^q$ via $\psi$ without mentioning it anymore.

By Theorem \ref{autonomo}, the vector field $H$ is an infinitesimal generator. Let  $(\varphi_t)$ be the associated semigroup.  Since $U\subset\subset M$, there exists $C\geq 0$ such that $\|H(z)\|\leq C$ for all $z\in U$. Also, by continuity, there exists $\mu>0$ such that for all $0\leq t<\mu$ the mapping $\varphi_t$ sends $\tfrac{1}{2}U$ in $U$, and thus
\begin{equation}\label{euclidean}
\|\varphi_t(z)-z\|\leq \int_0^t\|H(\varphi_\xi(z))\|d\xi\leq Ct,\quad z\in \tfrac{1}{2}U, \ 0\leq t\leq \mu.
\end{equation}

By Lemma \ref{abate} and \eqref{euclidean} there exists $\mathcal{L}\geq 0$ such that
$$k_M(\varphi_t(z),z)\leq \mathcal{L}t,\quad  \ z\in \tfrac{1}{2}U,\ 0\leq t<\mu.$$

Thus for all $0\leq \tau<\mu$ and all $z\in \frac{1}{2}U$, $\ell\in \N$
$$k_M(\varphi_\tau^{\circ\ell}(z),z)\leq \sum_{j=0}^\ell k_M(\varphi_\tau^{\circ(j+1)}(z),\varphi_\tau^{\circ j}(z))\leq \ell k_M(\varphi_\tau(z),z)\leq \ell \mathcal{L}\tau.$$

Fix $r>0$ and $t>0$ such that the Kobayashi ball $B_M(0, t\mathcal{L}+r)$ of center $0$ and radius $t\mathcal{L}+r$ is contained in  $\frac{1}{2}U$.
 Let $m\in \mathbb{N}$ be such that $t/m\leq \mu$.
Then for all $z\in {B_M(0,r)}$,
$$k_M(\varphi_{t/m}^{\circ \ell}(z),0)\leq k_M(\varphi_{t/m}^{\circ\ell}(z),z)+k_M(z,0)\leq \ell \mathcal{L}(t/m)+r,$$
hence $$\{\varphi_{t/m}^{\circ\ell}(z):0\leq \ell\leq m-1\}\subset B_M(0, t\mathcal{L}+r).$$

By (\ref{euclidean}), the family  $\frac{1}{h}(\varphi_h(z)-z)$ is bounded on  $\frac{1}{2}U$ and thus   converges uniformly to $H(z)$ as $h\to 0+$.
Up to shrinking $U$ if necessary, by hypothesis the same holds  for  $\frac{1}{h}(f_h(z)-z)$. Hence for each $\varepsilon >0$ there exists $\eta>0$ such that
$$k_M(\varphi_s(y),f_s(y))\leq s\varepsilon,\quad y\in \tfrac{1}{2}U,\ 0\leq s\leq \eta.$$

Take $N>0$ so large that $s:=\frac{t}{m}\leq \eta$ for all $m>N$. Then for all $z\in B_M(0,r)$ we have for $m>N$,
\begin{align}
k_M(f_{t/m}^{\circ m}(z),\varphi_t(z))&= k_M(f_{t/m}^{\circ m}(z),\varphi_{t/m}^{\circ m}(z))\nonumber \\
&\leq \sum_{k=1}^m k_M(f_s^{\circ(k-1)}(f_s(\varphi_s^{\circ(m-k)}(z))),f_s^{\circ(k-1)}(\varphi_s(\varphi_s^{\circ(m-k)}(z))))\nonumber\\
&\leq \sum_{k=1}^mk_M(f_s(y_k),\varphi_s(y_k))\leq t\varepsilon, \nonumber
\end{align}
where $y_k:=\varphi_s^{\circ(m-k)}(z)\in B_M(0, t\mathcal{L}+r)\subset\frac{1}{2}U.$
Thus $f_{t/m}^{\circ m}(z)\to \varphi_t(z)$ uniformly on $B_M(0,r)$ for $t$ small.

Fix now $T>0$ and $w\in M$ and consider the compact curve $\Gamma=\{\v_t(w):t\in[0,T]\}$. The previous argument implies that there exists $\delta>0$ and an open neighborhood $Y$ of $\Gamma$ such that  $f_{t/m}^{\circ m}(w)\to \varphi_t(w)$ uniformly on $Y$ for all $0\leq t\leq \delta$. Let $u\in \N$ be so large that $T/u<\delta$. Then for all $0\leq t\leq T$ one has $$\v_t(w)=\v_{t/u}^{\circ u}(w)=\lim_{m\to\infty}f_{\frac{t}{mu}}^{\circ mu}(w).$$
 Fix now $0<p<u$ and consider integers of the form $j=mu+p$, with $m\in \N$. Set $$t_m:= \frac{mu}{mu+p}T .$$ Then $$f_{T/j}^{\circ j}(w)=f_{\frac{t_m}{mu}}^{\circ mu}(f_{\frac{t_m}{mu}}^{\circ p}(w))\to \v_T(w),$$ by uniformity in $t$. This implies
$$ \v_T(w)=\lim_{m\to\infty}f_{\frac{T}{m}}^{\circ m}(w).$$
\end{proof}

For a bounded convex domain of a complex Banach space, the following corollary is proved in \cite[Corollary 4]{reich-shoikhet}, and its proof is an easy consequence of Theorem \ref{autonomo} and Proposition \ref{product}.

\begin{corollary}
Let $M$ be a complete hyperbolic complex manifold.  Let $H_1$ and $H_2$ be two infinitesimal generators on $M$ with associated semigroups $(\varphi_t)$ and $(\psi_t)$, respectively. Then the holomorphic vector field $H_1+H_2$ is an infinitesimal generator and the associated semigroup $(\eta_t)$ satisfies
$$\eta_t=\lim_{m\to+\infty} (\varphi_{t/m}\circ \psi_{t/m})^{\circ m},$$ where the limit is uniform on compacta of $M$.\end{corollary}

We have also:

\begin{corollary}\label{cono}
Let $M$ be a complete hyperbolic complex manifold. The set of infinitesimal generators on $M$ form a closed convex cone with vertex in $0$.
\end{corollary}

\section{The Loewner equation on complete hyperbolic manifolds}

\begin{proposition}\label{A1}
Let $M$ be a complete hyperbolic manifold. Let $d\geq 1 $ and let $G(z,t)$ be a Herglotz vector field of order $d$. Then there exists a unique evolution family $(\varphi_{s,t})$ of order $d$ which solves the Cauchy problem \eqref{L-ODE}.
\end{proposition}
\begin{proof}

For $z\in M$ and $s\in \mathbb{R}^+$   let $\varphi_{s,t}(z)$ be the maximal solution with escaping time $I(s,z)>0$ which solves the Cauchy problem \eqref{L-ODE}.
Fix $z\neq w\in M$.
Let $$J:= [s,I(s,z))\cap [s,I(s,w))$$ and let  $$h(t) := k_M(\varphi_{s,t}(z), \varphi_{s,t}(w)).$$
The function $h(t)$ is locally absolutely continuous since by Lemma \ref{loclip} the Kobayashi distance $k_M$ is locally Lipschitz. Thus $h(t)$ is  differentiable for a.e. $t\in J$. By Theorem \ref{autonomo} for a.e. $t\in \mathbb{R}^+$ the holomorphic vector field $G(z,t)$ satisfies b), thus $\underbar{D}h(t)\leq 0$ for a.e. $t\in J$, which implies $$\frac{d}{dt}h(t)\leq 0,\quad \mbox{for a.e.}\ t\in J,$$ hence $h$ is non-increasing. If $I(s,z)<I(s,w)$, since $M$ is complete hyperbolic and $$\{\varphi_{s,t}(w)\}_{t\in [s,I(s,z)]}\subset\subset M,$$ we have $$+\infty=\limsup_{t\to I(s,z)}k_M(\varphi_{s,t}(z),\varphi_{s,t}(w))\leq k_M(z,w)<+\infty,$$ a contradiction. Then $I(s,z)\geq I(s,w)$, and similarly $I(s,w)\geq I(s,z)$.

The proof proceeds now as in Steps 2-6  of \cite[Proposition 1]{Bracci-Contreras-Diaz-II}.
\end{proof}

\begin{proposition}\label{A3}
Let $M$ be  a complete hyperbolic  manifold of dimension $q\geq 1$. Then for any evolution family $(\varphi_{s,t})$ of order $d\geq 1$ in $M$ there exists a Herglotz vector field of order $d\geq 1$ which satisfies \eqref{L-ODE}. Moreover, if $G'(z,t)$ is another weak holomorphic vector field which satisfies \eqref{L-ODE} then $G(z,t)=G'(z,t)$ for all $z\in M$ and a.e. $t\geq 0$.
\end{proposition}

\begin{proof}
The proof is similar to that of \cite[Proposition 2]{Bracci-Contreras-Diaz-II} and we only describe the main differences. Let $U$ be a chart in $M$, let $U'\subset\subset U$ be an open set, and let $T>0$. There exists $n(T,U')>0$ such that for all $n\in \N, n\geq n(T,U')$, $t\in [0,T]$ it holds
\[
\v_{t,t+\frac{1}{n}}(U')\subset U.
\]
We define locally in $U'$,
\[
G_{n,s}(z):=n(\varphi_{s,s+\frac{1}{n}}(z)-z),\quad t\in [0,T],\ n\geq n(U',T).
\]

Let $k_T:=k_{\overline{U'},T+1}\in
L^{d}([0,T+1],\mathbb{R})$ be the non-negative function given
by EF3. We extend $k_T$ to all of $\R$ by setting zero outside
the interval $[0,T+1]$. Then for $0\leq s\leq T$ and every
$n\in\mathbb{N}$, $n\geq n(T,U')$ and all $z\in U'$,
\begin{equation*}
n\|
\varphi_{s,s+\frac{1}{n}}(z) -z\|=n\|
\varphi_{s,s+\frac{1}{n}}(z) -\varphi_{s,s}(z)\| \leq
n\int_{s}^{s+1/n}k_T(\xi)d\xi\leq {\sf Max}_{k_T}(s),
\end{equation*}
where
\[
{\sf Max}_{k_T}(s):=\sup\left\{
\frac{1}{|I|}\int_{I}k_T(\xi)d\xi:\text{ }I\text{ is a closed
interval of the real line and }s\in I\right\}
\]
is the so-called maximal function associated to $k_T$. Since
 $k\in L^{1}(\mathbb{R},\mathbb{R})$, by
Hardy-Littlewood maximal theorem there exists a subset
$N(T)\subset\lbrack0,+\infty)$ of zero measure such that
${\sf Max}_{k_T}(s)<+\infty$ for every $s\in\lbrack0,T]\setminus
N(T).$ Let
\[
N:=\bigcup_{T\in \N}N(T).
\]
Note that $N$ is a set of measure zero in $[0,+\infty)$. Fix $s\in [0,T]\setminus N$. Then for all  $T\in \N$ there exists $C(T,U',s)>0$ such that for all  $z\in U'$ and $n\geq n(T,U')$
\begin{equation}\label{claima}
\sup_{n}\| G_{n,s}(z)\|
\leq C(T,U',s).
\end{equation}
Let
\[
\Gamma:[0,+\infty)\rightarrow 2^{{\sf Hol}(U',\C^q)},\text{
}s\mapsto\Gamma(s)=\left\{
\begin{array}
[c]{ll}%
{\sf ac}(G_{n,s}) & s\notin N,\\
\{0\} & s\in N,
\end{array}
\right.
\]
where ${\sf ac}(G_{n,s})$ denotes the accumulation points of the
sequence $\{G_{n,s}\}$ in the metric space
${\sf Hol}(U',\C^q)$. Note that, since ${\sf Hol}(U',\C^q)$
 is a metric space, $\Gamma(s)$ is a closed subset of
${\sf Hol}(U',\C^q)$ for every $s\geq 0$. By \eqref{claima} for every fixed $s\in [0,+\infty)\setminus N$, the family $\{G_{n,s}\}_{n}$ is uniformly
bounded, thus it has accumulation
points in ${\sf Hol}(U',\C^q)$, so that
$\Gamma(s)$ is not empty for any $s\geq 0$.

Now we are going to apply the following result:

\begin{theorem}
\label{seleccion medible} \cite[Theorem III.30, page 80]{Castaing-Valadier}
Let $(\Omega,\Sigma,\mu)\,\,$be a positive $\sigma$-finite complete measure
space, $[X,d]$\ a separable and complete metric space and $\Gamma$\ a
multifunction from $\Omega$\ to the subsets of $X.$\ Assume that:

(i) For every $\omega\in\Omega,$\ $\Gamma(\omega)$\ is a closed
non-empty subset of $X$.

(ii) For every $x\in X$\ and every $r>0,$\
$\{\omega\in\Omega:\Gamma (\omega)\cap
B(x,r)\neq\emptyset\}\in\Sigma.$\ (As usual, $B(x,r)$\ denotes
the open unit ball in $X$\ with center $x$\ and radius $r).$

Then $\Gamma$\ admits a measurable selector
$\sigma:\Omega\longrightarrow X$; namely, for every
$\omega\in\Omega,$\ we have $\sigma(\omega)\in \Gamma(\omega)$\
and the inverse image by $\sigma$\ of any borelian in $X$\
belongs to $\Sigma.$
\end{theorem}

We already saw that $\Gamma$ satisfies hypothesis
(i) of Theorem~\ref{seleccion medible}.

In order to check condition (ii) in Theorem \ref{seleccion
medible} for $\Gamma$, one can argue similar to the proof of \cite[Theorem 6.2]{BCD1}, thus we omit such details.

Therefore, the multifunction $\Gamma$ satisfies the hypotheses
of Theorem \ref{seleccion medible}.  Thus there exists a
measurable selector $\sigma
:[0,+\infty)\rightarrow {\sf Hol}(U', \C^q)$ for
$\Gamma$. We define
$G:U'\times\lbrack0,+\infty)\rightarrow\C^q $ by
\[
G(z,s):=\sigma\lbrack s](z),
\]  for $z\in U'$  and
$s\in [0, +\infty)$. Hence, for every $s\in\lbrack0,+\infty)\setminus N$
there exists a strictly increasing sequence $\{n_{k}(s)\}\subset \N$  such that, for all $z\in U',$
\begin{equation*}
G(z,s):=\lim_{k\to \infty}G_{n_{k}(s),s}(z),
\end{equation*}
and the convergence is uniform on $U'$. Then, by construction, $G(z,s)$ is a weak holomorphic vector field on $U'$. Moreover, for any $s\in [0,T+1]\setminus N$ and for all $z\in U'$ it follows
\[
\|G_{n_k(s),s} \|\leq
n_k(s)\int_{s}^{s+1/n_k(s)}k_T(\xi)d\xi.
\]
Passing to the limit for $k\to \infty$, we obtain that for almost all $s\in [0,T+1]$ it holds
\[
\|G(z,s) \|\leq k_T(s),
\]
proving that $G$ has order $d$.

Now we can argue as in the proof of \cite[Proposition 2]{Bracci-Contreras-Diaz-II}, see the steps from 6' to 8' pp. 959--960, and we see that $G(z,t)$ is a weak holomorphic vector field of order $d$ which satisfies \eqref{L-ODE}. It is then a Herglotz vector field from $a)$ of Theorem \ref{autonomo}.
\end{proof}

\section{Open problems}

\subsection{On the definition of evolution families} Let $(\v_{s,t})$ with $0\leq s\leq t<+\infty$ be a family of holomorphic self-maps of a complete hyperbolic manifold $M$ with satisfies EF1 and EF2 of Definition \ref{def-ev}.

\medskip

{\bf Question 1}: Are there simple conditions which guarantee that EF3 holds?

\medskip

In other words, what (if any) are the simplest conditions that imply that $(\v_{s,t})$ is an evolution family? For instance, in \cite{BCD1} it is proved that a family $(\v_{s,t})_{0\leq s\leq t}$ of holomorphic self-maps of the unit disc $\D$ satisfying EF1, EF2 and
\begin{enumerate}
\item[EF3'.] for any  $z\in \D$ and for any $T>0$ there exists a
non-negative function $k_{z,T}\in L^{d}([0,T],\mathbb{R})$
such that for all $0\leq s\leq u\leq t\leq T$
\[
d_M(\varphi_{s,u}(z),\varphi_{s,t}(z))\leq\int_{u}^{t}k_{z,T}(\xi)d\xi.
\]
\end{enumerate}
solves \eqref{L-ODE} for a given $L^d$-Herglotz vector field $G$ on $\D$. Hence, {\sl a posteriori}, from Theorem \ref{main}, it follows that $(\v_{s,t})$ is a $L^d$-evolution family in the sense of Definition \ref{def-ev} (in fact, this can be proved directly using distortion theorems).

If $M=D\subset\subset \C^q$ is a bounded convex domain, a simpler condition can be stated as follows:

\begin{proposition}\label{convexLd}
Let $D\subset\subset \C^q$ be a bounded balanced convex domain, $d\in [1,+\infty]$. Let $(\v_{s,t})_{0\leq s\leq t}$ be a family of holomorphic self-maps of $D$ satisfying EF1 and EF2. Then $(\v_{s,t})$ is an $L^d$-evolution family if and only if for any $T>0$ there exists  a non-negative function $k_T\in L^{d}([0,T],\mathbb{R})$
such that for all $0\leq s\leq t\leq T$
\begin{equation}\label{cond1}
\|\v_{s,t}(0)\|+\|d(\v_{s,t})_0-{\sf id}\|\leq\int_{s}^{t}k_T(\xi)d\xi.
\end{equation}
\end{proposition}
The proof is based on some type of distortion theorem for infinitesimal generators. Let $A: \C^n \to \C^n$ be a continuous linear operator. As customary, we define
\[
V(A):=\sup \{|\langle A(v), v\rangle| : \|v\|=1\}.
\]

The following result is proved in its generality for the case of Banach spaces in \cite{Br-Sh}. In case $G(0)=0$ and $V(T)> 0$ it is proved in \cite{GHK2002}.

\begin{lemma}\label{Br-Shthm}
Let $D\subset \C^n$ be a bounded convex domain. Let $G(z)=G(0)+Tz+\sum_{j\geq 2} Q_j(z)$ be an infinitesimal generator in $D$. Then for all $z\in D$
\begin{equation}\label{estimate}
\|G(z)\|\leq 5\|G(0)\|+ \frac{4\|z\|}{(1-\|z\|)^2}V(T).
\end{equation}
\end{lemma}

\begin{proof}[Proof of Proposition \ref{convexLd}]
One direction is clear. As for the other, assume that \eqref{cond1} holds and fix $T>0$. Let $G_{n,s}(z):=n(\varphi_{s,s+\frac{1}{n}}(z)-z)$ for $n\in \N$ and $z\in D$. By \cite[Proposition 4.3]{reich-shoikhetII}, $G_{n,s}(z)$ is an infinitesimal generator in $D$. Hence Lemma \ref{Br-Shthm} and \eqref{cond1} implies \eqref{claima}. From here we can argue similarly to the proof of Proposition \ref{A3} and obtain an $L^d$-Herglotz vector field $G(z,s)$ which satisfies \eqref{L-ODE}. Then by Theorem \ref{main} the family $(\v_{s,t})$ is an $L^d$-evolution family.
\end{proof}

\subsection{Relations between semigroups and evolution families} Let $G(z,t)$ be a Herglotz vector field on a bounded convex domain $D$ with associated evolution family $(\v_{s,t})$. Hence, by \eqref{L-ODE}, for a.e. $s\geq 0$, it holds
\[
G(z,s)=\lim_{t\to s^+} \frac{\v_{s,t}(z)-z}{t-s}.
\]
In particular, if $(\phi^s_r)_{r\geq 0}$ is the semigroup associated to the infinitesimal generator $z\mapsto G(z,s)$, the product formula implies that
\[
\phi^s_r=\lim_{m\to \infty} (\v_{s,s+\frac{r}{m}})^{\circ m}
\]
uniformly on compacta of $D$. Thus, there is a link between the families of semigroups generators by $G(\cdot, s)$ when $s\geq 0$ and the evolution family given by \eqref{L-ODE}. Thus we have the following natural philosophical question:

\medskip

{\bf Question 2}: What are the relations between the dynamics of the families of semigroups $(\phi^s_r)$ and the dynamics of the evolution family $(\v_{s,t})$?

\medskip

The previous question is open even for the case of the unit disc $\D\subset \C$.

\subsection{The embedding problem} Let $f: M \to M$ be an univalent self-map of a complete hyperbolic manifold.

\medskip

{\bf Question 3}: When does there exists an evolution family $(\v_{s,t})$ order $d\geq 1$ such that $f=\v_{0,1}$?


\begin{thebibliography}{99}

\bibitem{Abate} M. Abate, \textit{Iteration Theory of Holomorphic Maps on Taut
Manifolds}, Mediterranean Press, Rende, Cosenza, (1989).

\bibitem{Abate-generators} M. Abate, \textit{The infinitesimal generators of semigroups of holomorphic maps}, Ann. Mat. Pura Appl. (4) \textbf{161} (1992), 167--180.

\bibitem{ABCD} M. Abate, F. Bracci, M. D. Contreras, S. D\'iaz-Madrigal, {\sl The evolution of Loewner's differential
equations}, Newsletter European Math. Soc. \textbf{78}  (2010), 31--38.

\bibitem{ABHK} L. Arosio, F. Bracci, H. Hamada, G. Kohr, { \sl Loewner's theory on complex
manifolds}, Preprint 2010.


\bibitem{AMR} R. Abraham, J.E. Marsden, T. Ratiu, {\sl Manifolds, Tensor Analysis, and Applications},  2nd edn. Applied Mathematical Sciences, vol. 75. Springer-Verlag, New York (1988)


\bibitem{Br} F. Bracci, {\sl Holomorphic evolution: metamorphosis of the Loewner equation}, J. Anal. Math, to appear.


\bibitem{BCM} F. Bracci, M.D. Contreras, S. D\'iaz-Madrigal, \textit{ Pluripotential theory,
semigroups and boundary behavior of infinitesimal generators in
strongly convex domains}, J. Eur. Math. Soc. \textbf{12} (2010), 23--53.


\bibitem{BCD1} F. Bracci, M.D. Contreras, S.
D\'{\i}az-Madrigal, \textit{Evolution Families and the Loewner Equation I: the unit disc}, J. Reine Angew. Math. (Crelle's Journal), 672, (2012), 1-37.


\bibitem{Bracci-Contreras-Diaz-II}
F. Bracci, M. D. Contreras and S. D\'iaz-Madrigal, \textit{Evolution
families and the Loewner equation II: complex hyperbolic manifolds},
Math. Ann. \textbf{344} (2009), 947--962.

\bibitem{Br-Sh} F. Bracci, D. Shoikhet, {\sl Growth estimates for infinitesimal generators in Banach spaces}, in preparation.



\bibitem{Castaing-Valadier}C. Castaing and M. Valadier, Convex Analysis and
Measurable Multifunctions, Lecture Notes in Math. Vol. 580, Springer-Verlag,
Berlin, 1977.

\bibitem{CDG} M.D. Contreras, S. D\'{\i}az-Madrigal,  P. Gumenyuk,
\textit{Loewner chains in the unit disc},  Rev. Mat. Iberoamericana \textbf{26} (2010), 3, 975--1012.



\bibitem{GHK2002} I. Graham, H. Hamada, G. Kohr, {\sl Parametric representation of univalent mappings in several complex variables}. Canad. J. Math. \textbf{54} (2002), no. 2, 324--351.


\bibitem{G-K} I. Graham, G. Kohr, {\sl Geometric function
theory in one and higher dimensions}, Marcel Dekker Inc., New
York, 2003.


\bibitem{HKM} H. Hamada, G. Kohr, J. R. Muir, {\sl Extension of $L^d$-Loewner chains to higher dimensions}, Preprint 2011.


\bibitem{Hormander}
L. H\"ormander, \textit{An introduction to complex analysis in several variables}. D. Van Nostrand Co., Inc., Princeton, N.J.-Toronto, Ont.-London 1966


\bibitem{Kob} S. Kobayashi, {\sl Hyperbolic complex spaces}. Springer-Verlag Berlin Heidelberg, 1998

\bibitem{Le} L. Lempert, {\sl La m\'etrique de Kobayashi et
la  representation des domaines sur la boule}. Bull. Soc. Math.
France \textbf{109}  (1981), 427--474.

\bibitem{reich-shoikhetII}
S. Reich and D. Shoikhet, \textit{Generation theory for semigroups of holomorphic mappings in Banach spaces} (English summary)
Abstr. Appl. Anal. \textbf{1} (1996), no. 1, 1--44


\bibitem{reich-shoikhet}
S. Reich and D. Shoikhet, \textit{Metric domains, holomorphic mappings and nonlinear semigroups} Abstr. Appl. Anal. \textbf{3} (1998), no. 1-2, 203--228





\bibitem{RS} S. Reich, D. Shoikhet, {\sl Nonlinear Semigroups, Fixed Points,
and Geometry of Domains in Banach Spaces}, Imperial College Press, London, 2005.

\end{thebibliography}
\end{document}